\documentclass[12pt,twoside,final,psamsfonts]{amsart}
\usepackage{amsfonts}
\usepackage{amsmath}
\usepackage{amssymb}
\usepackage{amsthm}
\usepackage[shortlabels]{enumitem}
\usepackage{relsize}
\usepackage{ mathrsfs }

\newtheorem{theorem}{Theorem}
\newtheorem{lemma}{Lemma}
\newtheorem{corollary}{Corollary}
\newtheorem{definition}{Definition}

\newcommand{\Z}{\mathbb{Z}}
\newcommand{\D}{$\mathbb{D}$}
\newcommand{\DD}{\mathbb{D}}
\newcommand{\N}{$\mathcal{N}$}

\newenvironment{proof1}{\paragraph{Proof of Theorem \ref{teorema1}}}{\hfill$\square$}
\newenvironment{proof2}{\paragraph{Proof of Theorem \ref{teorema2}}}{\hfill$\square$}
\newenvironment{proof3}{\paragraph{Proof of Theorem \ref{teorema3}}}{\hfill$\square$}
\newenvironment{proof4}{\paragraph{Proof of Theorem \ref{teorema4}}}{\hfill$\square$}

\newenvironment{proofc5}{\paragraph{Proof of Corollary 2}}{\hfill$\square$}

\title{Invertibility threshold for Nevanlinna quotient algebras}
\author{ Artur Nicolau and Pascal J. Thomas }

\address{A. Nicolau: Universitat Aut\`onoma de Barcelona\\
Departament de Matem\`a\-tiques\\
 08193-Barcelona\\ Catalonia}
\email{artur@mat.uab.cat}

\address{P. Thomas:  
Institut de Mathï¿½matiques de Toulouse ; UMR5219 \\
Universit\'e de Toulouse ; CNRS \\
UPS IMT, F-31062 Toulouse Cedex 9 \\
France}
 \email{pascal.thomas@math.univ-toulouse.fr}


\thanks{First author is supported by the Generalitat de Catalunya (grant 2017 SGR 395) and the Spanish Ministerio de Ciencia e Innovaci\'on (project MTM2017-85666-P.)}

\begin{document}

\begin{abstract} 
Let $\mathcal{N}$ be the Nevanlinna class and let $B$ be a Blaschke product. It is shown that the natural invertibility criterion in the quotient algebra  $\mathcal{N} / B \mathcal{N}$, that is, $|f| \ge e^{-H} $ on the 
set $B^{-1}\{0\}$ for some positive harmonic function $H$, holds if and only if the function $- \log |B|$ has a harmonic majorant on 
 the set $\{z\in\mathbb{D}:\rho(z,\Lambda)\geq e^{-H(z)}\}$; at least
 for large enough functions $H$. We also study the corresponding class of positive harmonic functions $H$ in the unit disc such that the latter
 condition holds. 
 We also discuss the analogous invertibility problem in quotients of the Smirnov class. 
\end{abstract}

\maketitle



\section{Introduction}
The Nevanlinna class $\mathcal{N}$ is the algebra of analytic functions $f$ in \D,
the unit disc of the complex plane,  such that $\log|f|$ has a positive harmonic majorant in \D. 
Any $f\in \mathcal{N}$ factors as $f=B F$, where $B$ is a Blaschke product and $F\in \mathcal{N}$  is 
 invertible, as are all zero-free functions in the Nevanlinna class. Any principal ideal in $\mathcal{N}$ is thus of the form $B\mathcal{N}$, and if the zero set of $B$
is $\Lambda=\{\lambda_k\}$, this ideal is the set of functions in $\mathcal{N}$ which vanish at $\Lambda$.
Fix a Blaschke product $B$ and $f\in \mathcal{N}$. It is clear that  the class $[f]=\{f+Bh:h\in\mathcal{N}\}$ is uniquely determined by the
restriction of $f$ to $\Lambda$. If $[f]$ is
invertible in the quotient algebra $\mathcal{N} \big/ B\mathcal{N}$, then there exists a positive harmonic
function
$H$ such that $|f(\lambda_k)|\geq e^{-H(\lambda_k)}$, $k=1,2,\dots$ However, the converse is not true in general. 
For a given Blaschke product $B$, we would like to find out which positive harmonic functions $H$
will make the converse true.

The analogous problem for the algebra $H^\infty$  of bounded analytic functions $f$ in the unit disc,
with the obvious necessary condition for invertibility
$|f(\lambda_k)|\geq \varepsilon >0$,
was studied in \cite{GMN} in connection with the Corona Theorem and interpolating sequences.  We first need to give some background on the classical $H^\infty$ theory. Recall that $H^\infty$ is endowed with the norm $||f||_\infty:=\sup\{|f(z)|:z\in$ \D $\}$.

We will use the standard \emph{pseudohyperbolic distance} on $\mathbb{D}$ given by 
\[ \rho(a_1 ,  a_2 ) := |a_1 - a_2| |1 - \overline{a_2} a_1|^{-1} ,
\]
for $a_1, a_2 \in \DD$.  

We now recall two classical results of Carleson. A sequence of points $\Lambda=\{\lambda_k\}$ in $\mathbb{D}$ is called an 
\emph{interpolating sequence} (for $H^\infty$) if for any bounded sequence of complex values $\{w_k\}$ there exists $f\in H^\infty$ such that $f(\lambda_k)=w_k$, $k=1,2,\ldots$. A celebrated result of Carleson states that $\Lambda=\{\lambda_k\}$ is interpolating if and only if there exists a constant $\delta >0$ such that
\begin{equation*}
    (1-|\lambda_k|^2)|B'(\lambda_k)|\geq\delta,\hspace{1cm} k=1,2,\ldots,
\end{equation*}
where $B$ is the Blaschke product with zeros $\{\lambda_k\}$. Observe that 
$$(1-|\lambda_k|^2)|B'(\lambda_k)| =  \prod_{j:j\neq k} \rho(\lambda_k,\lambda_j) . $$
The classical Corona Theorem states that the ideal generated by the functions $f_1,\ldots,f_n\in H^\infty$ is the whole algebra $H^\infty$ if and only if $\inf\{|f_1(z)|+\cdots+|f_n(z)|:z\in$ \D $\}>0$. See \cite{Ca}, \cite{Ca2}, or Chapters VII and VIII of \cite{Ga}.

 A function $I\in H^\infty$ is called \emph{inner} if 
$\left|\lim_{r\to1} I(r\xi)\right|=1$ for almost every $\xi$ in the unit circle $\partial$\D. 
Any inner function $I$ factors as $I=B S$, where $B$ is a Blaschke product and $S$ is an inner function without zeros. It follows from the classical Theorem of Beurling on the invariant subspaces of the shift operator, that any weak* closed ideal in $H^\infty $ is of the form $I H^\infty = \{Ih : h \in H^\infty \}$ for some inner function $I$. See \cite{Ga}, page 82. Fix an inner function $I$ and consider the quotient Banach algebra $H^\infty \big/ IH^\infty$ with the norm 
\begin{equation*}
||[f]||_{H^\infty \big/ IH^\infty}=\inf\{||f+Ih||_\infty:h\in H^\infty\},\hspace{0.5cm}f\in H^\infty.
\end{equation*}
Let $\Lambda=\{\lambda_k\}$ be the zero set of $I$. It is clear that if $[f]$ is invertible in $H^\infty \big/ IH^\infty$, then $\inf|f(\lambda_k)|>0$. This condition is not always sufficient as one can observe considering the extreme situation where $I$ is zero-free. 

Let $\delta=\delta(I)$ be the infimum of the positive numbers $\gamma>0$ such that if $f\in H^\infty$, $||f||_\infty\leq1$ satisfies $\inf_k|f(\lambda_k)|\geq\gamma$, then $[f]$ is invertible in $H^\infty \big/ IH^\infty$, or equivalently, there exist $g,h\in H^\infty$ such that $fg=1+Ih$. Hence, if $\gamma>\delta(I)$, for any $f\in H^\infty$, $||f||_\infty\leq1$ with $\inf_k|f(\lambda_k)|\geq\gamma$, we have that $[f]$ is invertible in $H^\infty \big/ IH^\infty$; while if $0<\gamma<\delta(I)$ there exists $f\in H^\infty$, $||f||_\infty\leq1$ with $\inf_k|f(\lambda_k)|\geq\gamma$ such that $[f]$ is not invertible in $H^\infty \big/ IH^\infty$. 

Gorkin, Mortini and Nikolski proved in \cite{GMN} that $\delta(I)=0$ if and only if $I$ satisfies that, for any $\varepsilon>0$, we have
\begin{equation}
    \inf\{|I(z)|:\rho(z, \Lambda )>\varepsilon\} >0 .
\label{eq1}
\end{equation}

If $I$ is a Blaschke product whose zeros $\Lambda$ are a finite union of interpolating sequences (or, equivalently, if $\sum_k(1-|\lambda_k|)\delta(\lambda_k)$ is a Carleson measure), then condition (\ref{eq1}) is satisfied. Here $\delta(\lambda_k)$ denotes the Dirac mass at the point $\lambda_k$. However, there are Blaschke products $I$ satisfying (\ref{eq1}) whose zeros are not a finite union of interpolating sequences. See \cite{GMN} and \cite{Bo}. For this reason, the authors of \cite{GMN} called property (\ref{eq1}) the \emph{Weak Embedding Property (WEP)}. It would be interesting to describe the Blaschke products $I$ satisfying the WEP in terms of the location of their zeros. Some further results and questions on inner functions satisfying the WEP can be found in \cite{BNT}.

The study of the invertibility in $H^\infty \big/ IH^\infty$ was continued by Nikolskii and Vasyunin in \cite{NV}, where it was proved that for any $0<\delta<1$, there exists a Blaschke product $I$ such that $\delta(I)=\delta$. 
In other words, one can find an invertibility threshold at any level, by choosing
the Blaschke product appropriately.  The main purpose of this paper is to discuss the analogous problem in the Nevanlinna class.

We now turn to the analogues in the Nevanlinna class of the above results. 
In many of those, positive harmonic functions  will play the role that was played by positive constants
in the $H^\infty$ setting. We begin with interpolating sequences.

Let Har$_+($\D$)$ denote the cone of positive harmonic functions in \D. Given a sequence of points $\Lambda=\{\lambda_k\}\subset$ \D, let $W(\Lambda)$ be the set of sequences of complex sumbers $\{w_k\}$ such that there exists $H\in$ Har$_+($\D$)$ with $\log^+|w_k|\leq H(\lambda_k)$, $k=1,2,\ldots$. Observe that $\{f(\lambda_k)\}\subset W(\Lambda)$ for any $f\in \mathcal{N}$. A sequence of points $\Lambda=\{\lambda_k\}\subset$ $\mathbb{D}$ is called an \emph{interpolating sequence} for $\mathcal{N}$ if for any sequence of values $\{w_k\}\subset W(\Lambda)$ there exists $f\in\mathcal{N}$ such that $f(\lambda_k)=w_k,$ $k=1,2,\dots$ It was proved in \cite{HMNT} that $\Lambda=\{w_k\}$ is an interpolating sequence for $\mathcal{N}$ if and only if there exists $H\in$ Har$_+$(\D) such that
\begin{equation}
\label{NIS}
    (1-|\lambda_k|^2)|B'(\lambda_k)|\geq e^{-H(\lambda_k)},\hspace{0.5cm} k=1,2,\ldots,
\end{equation}
where $B$ is the Blaschke product with zeros $\{\lambda_k\}$. 

Using a result of T. Wolff, Mortini proved the following version of the Corona Theorem for $\mathcal{N}$. Let $f_1,\ldots,f_n\in$ \N. Then there exist $g_1,\ldots,g_n\in\mathcal{N}$ such that $f_1g_1+\cdots+f_ng_n=1$ if and only if the function $-\log(|f_1|+\cdots+|f_n|)$ has a harmonic majorant in $\mathbb{D}$. See \cite{M} or \cite{Ma}. 

The analogue of the WEP in the Nevanlinna class was introduced in \cite{MNT} where it was proved that invertible classes $[f]$ in $\mathcal{N} \big/ B\mathcal{N}$ are precisely the classes for which there exists $H=H(f)\in$ Har$_+(\mathbb{D})$ such that $|f(\lambda_k)|\geq e^{-H(\lambda_k)}$, $k=1,2,\ldots$ if and only if $B$ satisfies the following analogue of the WEP: for any $H_1\in$ Har$_+(\mathbb{D})$ there exists $H_2\in$ Har$_+(\mathbb{D})$ such that
\begin{equation*}
    |B(z)|\geq e^{-H_2(z)}\hspace{0.5cm}\text{ if }\hspace{0.5cm}\rho(z, \Lambda )\geq e^{-H_1(z)}.
\end{equation*}
In contrast with the situation in $H^\infty$, the main result in \cite{MNT} states that this property holds if and only if the zeros of $B$ are a finite union of interpolating sequences in the Nevanlinna class.\\

As we said, the main purpose of the present paper 
is to discuss an analogue in the Nevanlinna class of the result of Nikolski 
and Vasyunin (\cite{NV}) described above. Let $B$ be the Blaschke product with zero set $\Lambda=\{\lambda_k\}$. Fix $f\in\mathcal{N}$. As mentioned above, a necessary condition for the class $[f]$ to be invertible in $\mathcal{N} \big/ B\mathcal{N}$ is that there exists $H\in$ Har$_+($\D$)$ such that 
\begin{equation}
    |f(\lambda_k)|\geq e^{-H(\lambda_k)}, \hspace{0.5cm} k=1,2,\ldots.
    \label{eq2}
\end{equation}
In analogy with the definition of $\delta(B)$ in the context of $H^\infty$, we are interested in which 
functions $H\in$ Har$_+($\D$)$ have the property that (\ref{eq2}) guarantees
 that the class $[f]$ is invertible in $\mathcal{N} \big/ B\mathcal{N}$. Since functions $f\in\mathcal{N}$ without zeros are invertible in $\mathcal{N}$, we can always choose a representative of any class
 $[f]$ as
 $f\in H^\infty$ with $||f||_{\infty}\leq 1$ and 
 assume that $H$ is bigger than any prescribed function in Har$_+(\mathbb{D})$. Our first result says that this invertibility problem is equivalent to the existence of a  harmonic majorant
 of $-\log |B|$ restricted to a certain subset of $\DD$. 

\begin{definition}
A function $F\colon \mathbb{D}\longrightarrow[0,+\infty)$ has a \emph{harmonic majorant} on the set $E\subset\mathbb{D}$ if there exists $H\in$ Har$_+($\D$)$ such that $F(z)\leq H(z)$ for any $z\in E$. 
\end{definition}

We will need an auxiliary function associated to any Blaschke sequence. 

\begin{definition}
Given a Blaschke sequence $\Lambda=\{\lambda_k\}$, let $H_\Lambda$ denote the positive harmonic function defined by
\begin{equation}
\label{funcio}
H_\Lambda(z)=\sum_k\int_{I_k}\frac{1-|z|^2}{|\xi-z|^2}|\mathrm{d}\,\xi|,\hspace{0.5cm}z\in\mathbb{D},
\end{equation}
where $I_k:=\{\xi\in\partial\mathbb{D}:|\xi-\lambda_k / |\lambda_k||\leq 1-|\lambda_k|\}$ denotes the Privalov 
shadow of $\lambda_k$.
\end{definition}

\begin{theorem}
\label{teorema1}
Let $B$ be a Blaschke product with zero set $\Lambda=\{\lambda_k\}$.
    \begin{enumerate}[(a)]
        \item There exists a universal constant $C>0$ such that the following statement holds. Let $H\in$ Har$_+($\D$)$ and assume that the function $-\log|B|$ has a harmonic majorant on the set $\{z\in\mathbb{D}:\rho(z,\Lambda)\geq e^{-H(z)}\}$. Then for any $f\in H^\infty$, $||f||_\infty\leq1$ such that
            \begin{equation}
\label{parta}
             |f(\lambda_k)|> e^{-CH(\lambda_k)}, \quad k=1,2, \ldots , 
            \end{equation}
        there exist $g,h\in\mathcal{N}$ such that $fg=1+Bh$.
        \item There exist universal constants  $C_0>0$ and $C >0$ such that the following statement holds. Let $H\in$ Har$_+($\D$)$ with $H\geq C_0 H_\Lambda$. Assume that for any $f\in H^\infty$, $||f||_\infty\leq1$ such that
            \begin{equation*}
\label{partb}                
 |f(\lambda_k)|> e^{-CH(\lambda_k)}, \quad k=1,2, \ldots , 
            \end{equation*}
        there exist $g,h\in\mathcal{N}$ such that $fg=1+Bh$. Then, the function $-\log|B|$ has a harmonic majorant on the set $\{z\in\mathbb{D}\, : \rho(z,\Lambda)\geq e^{-H(z)}\}$.
    \end{enumerate}
\end{theorem}
In Corollary \ref{corona}, after the proof of this theorem in Section \ref{pf12},
we show how the result can be extended to B\'ezout equations with any
number of generators.

Observe that Theorem \ref{teorema1} is analogous to the equivalence of (a) and (b) in \cite[Theorem A]{MNT}. 
Hence, given a Blaschke product $B$
 with zero set $\Lambda$, and for large enough positive harmonic functions $H$,
 the invertibility problem in the quotient algebra 
 $\mathcal{N } / B \mathcal{N}$
 can be reduced to the study of the following class.
 
\begin{definition}
Given a Blaschke product $B$, let $\mathcal{H}(B)$ be the set of functions $H\in$ Har$_+(\mathbb{D})$ such that $-\log|B|$ has a harmonic majorant on the set $\{z\in\mathbb{D}\,:\rho(z,\Lambda)\geq e^{-H(z)}\}$. 
 \end{definition}
 
  It is easy to see that constant functions are always in $\mathcal{H}(B)$ (see Proposition 4.1 of \cite{HMNT} or Lemma \ref{lema1} below), and that if $H_1\in \mathcal{H}(B)$ and $H_2 \in \mbox{Har}_+(\mathbb{D})$,
 $H_2 \le H_1$, then $H_2\in \mathcal{H}(B)$. In this language the main result of \cite{MNT} reads as follows: $\mathcal{H}(B)=$ Har$_+(\mathbb{D})$ if and only if $\Lambda$ is a finite union of interpolating sequences for $\mathcal{N}$.

 Our next result says that for any Blaschke product $B$, $\mathcal{H}(B)$ does contain unbounded functions. 
 \begin{theorem} 
 \label{teorema2}
 Let $B$ be a Blaschke product with zero set $\Lambda=\{\lambda_k\}$.Then,
\begin{enumerate}[(a)]
        \item There exists a function $H\in$ Har$_+(\mathbb{D})$ with $\limsup_{k\to\infty} H(\lambda_k)=+\infty$, such that $-\log|B|$ has a harmonic majorant on the set $\{z\in\mathbb{D}\,:\rho(z,\Lambda)\geq e^{-H(z)}\}$.
        \item There exists a function $H\in$ Har$_+(\mathbb{D})$ with $\limsup_{k\to\infty} H(\lambda_k)=+\infty$ such that if $f\in H^\infty$, $||f||_\infty\leq1$ satisfies $|f(\lambda_k)|\geq e^{-H(\lambda_k)}$, $k=1,2,\ldots$, then there exist $g,h\in\mathcal{N}$ such that $fg=1+Bh$.
\end{enumerate}
 \end{theorem}

Conversely, given two positive harmonic functions $H_1, H_2$, where the condition
$H_1 \le H_2$ does not hold, we would like to see 
whether there exists a Blaschke product that discriminates between them, that is to say,
such that $H_2\in \mathcal{H}(B)$ but $H_1\notin \mathcal{H}(B)$.  We obtain
such a Blaschke product in two different cases.

\begin{theorem}
\label{teorema5}
\begin{enumerate}[(a)]
\item
Let $H_1,H_2\in$ Har$_+($\D$)$ such that
\begin{equation*}
    \limsup_{|z|\to1}\;\frac{H_1(z)}{H_2(z)}=+\infty.
\end{equation*}
Then there exists a Blaschke product $B$ with zero set $\Lambda$ such that $-\log|B|$ has a harmonic majorant on the set $\{z\in\mathbb{D}:\rho(z,\Lambda)\geq e^{-H_2(z)}\}$ but has no harmonic majorant on the set $\{z\in\mathbb{D}:\rho(z,\Lambda)\geq e^{-H_1(z)}\}$.
\item
For any $\eta_0>0$, and any unbounded positive harmonic function $H$, there exists a Blaschke product $B$ such that $H \in \mathcal H(B)$ but $(1+\eta_0)H \notin \mathcal H(B)$.
\end{enumerate}
\end{theorem}

The first part of the theorem can be applied in particular
when $H_2=1$, which means that for any unbounded $H_1 \in \mbox{Har}_+(\DD)$,
there exists a Blaschke product $B$ so that $H_1 \notin \mathcal H(B)$. It should be noted that in the second part of the theorem, 
the Blaschke product $B$ has zeros concentrated in a way controlled by the size
of the harmonic function $H$ we started from.  The next result involves
this critical size.  In order to state it, we need more notation.
 
Consider the usual dyadic Whitney squares
\begin{equation*}
    Q_{k,j}=\{re^{i\theta}\in\mathbb{D}\,:2^{-k}\leq1-r<2^{-k+1}, j 2\pi  2^{-k}\leq\theta<(j+1)2\pi  2^{-k}\}.
\end{equation*}
where $k\geq0$ and $j=0,\ldots,2^k-1$. Consider also the corresponding projections on $\partial\mathbb{D}$ given by
\begin{equation*}
    I_{k,j}=\{e^{i\theta}\in\partial\mathbb{D} : j 2\pi 2^{-k}\leq\theta<(j+1)2\pi 2^{-k}\}.
\end{equation*}

Given a Blaschke sequence $\Lambda=\{\lambda_k\}$ and a dyadic Whitney square $Q$ let $N(Q)=\#(\Lambda \cap Q)$ be the number of points of $\Lambda$ in $Q$. Observe that there exists a universal constant $C>0$ such that for any dyadic Whitney square $Q$ and any $z \in Q$ we have $H_ \Lambda (z) \geq C N(Q)$.

In connection to part (b) of Theorem \ref{teorema5}, it is interesting to observe that for functions 
$H\in$ Har$_+($\D$)$ growing sufficiently fast 
with respect to the number of zeros of $B$,
we have $H\in\mathcal{H}(B)$ if and only if $CH\in\mathcal{H}(B)$ for some (all) constants $C>0$.

\begin{theorem}
\label{teorema3}
Let $B$ be a Blaschke product with zero set $\Lambda$. Let $H\in$ Har$_+($\D$)$ such that 
\begin{equation}
\label{control}
\inf\{e^{H(z)}:z\in Q\}\geq N(Q),
\end{equation}
for any dyadic Whitney square $Q$. Assume that the function $-\log|B|$ has a harmonic majorant on the set $\{z\in\mathbb{D}:\rho (z,\Lambda)\geq e^{-H(z)}\}$. Then, for any $C>1$, the function $-\log|B|$ has a harmonic majorant on the set $\{z\in\mathbb{D}:\rho (z,\Lambda)\geq e^{-CH(z)}\}$.
\end{theorem}

Part (b) of Theorem \ref{teorema5} shows that the above result no longer holds when condition \eqref{control}
is not verified.

Our last result provides a sufficient condition for a function $H\in$ Har$_+($\D$)$ to belong to $\mathcal{H}(B)$. Given a dyadic Whitney square $Q$ let $z(Q)$ denote its center. 

\begin{theorem}
\label{teorema4}
Let $B$ be a Blaschke product with zero set $\Lambda$. Let $\mathcal{A}$ be the collection of dyadic Whitney squares $Q$ such that $N(Q)=\#(\Lambda\cap Q)>0$. Let $H\in$ Har$_+($\D$)$. Assume that there exists $H_1\in$ Har$_+($\D$)$ such that $N(Q)H(z(Q))\leq H_1(z(Q))$ for any $Q\in\mathcal{A}$. Then, the function $-\log|B|$ has a harmonic majorant on the set $\{z\in\mathbb{D}:\rho(z,\Lambda)\geq e^{-H(z)}\}$.
\end{theorem}
Notice that we impose no direct restriction on the values of $H$ in the dyadic squares where no zero of $B$ 
is present. Moreover, we will introduce a class of Blaschke products for which this sufficient condition is also necessary. This is done at the end of Section 2. In Section 4 we study the corresponding invertibility problem in quotients of the Smirnov class.

In this paper, $C_0, C_1, C_2 \ldots$ will denote absolute constants while $C(\delta)$ will denote a constant which depends on the parameter $\delta>0$.

This work was motivated by a question asked to one of us by Nikolai Nikolski during the Congress "Complex Analysis and Related Topics" in April 2018, about the Nevanlinna analogue of the main result of \cite{NV}. It is also a pleasure to thank Xavier Massaneda for many helpful conversations.  

\section{Theorems 1 and 2}
\label{pf12}
Recall that if $\Lambda$ is a Blaschke sequence, $H_\Lambda$ denotes the positive harmonic function defined in (\ref{funcio}). The proof of Theorem \ref{teorema1} uses two auxiliary results. The first one is Proposition 4.1 of \cite{HMNT}.

\begin{lemma}
\label{lema1}
Let $B$ be a Blaschke product with zero set $\Lambda$. Then, for any $\delta>0$ there exists $C=C(\delta)>0$ such that $-\log|B(z)|\leq CH_\Lambda(z)$ if $\rho(z,\Lambda)\geq\delta$.
\end{lemma}

\begin{lemma}
\label{lema2}
Let $\Lambda$ be a Blaschke sequence and let $A$ be a sequence of points in $\mathbb{D}$ verifying that there exists a constant $0<\gamma<1$ and a natural number $k$ such that for any $a\in A$ there is $\lambda(a)\in\Lambda$ with $\rho(a,\lambda(a))\leq\gamma$ and $\#\{a\in A:\lambda(a)=\lambda\}\leq k$ for any $\lambda\in\Lambda$. Then, $A$ is a Blaschke sequence and for any $\delta>0$, there is a constant $C=C(\gamma,\delta,k)>0$ such that
\begin{equation*}
    \mathlarger{\mathlarger{\sum_{a\in A:\rho(a,z)>\delta}}}\log\left|\frac{a-z}{1-\overline{a}z}\right|^{-1}\leq CH_\Lambda(z),\hspace{0.5cm}z\in\mathbb{D} . 
\end{equation*}
\end{lemma}

\begin{proof}
Since $\rho(a,z)>\delta$, there is a constant $C_1=C_1(\delta)>0$ such that
\begin{equation*}
    \log\left|\frac{a-z}{1-\overline{a}z}\right|^{-1}\leq C_1 \hspace{1pt}\left(1-\left|\frac{a-z}{1-\overline{a}z}\right|^{2}\right)=C_1 \hspace{1pt}\frac{(1-|a|^2)(1-|z|^2)}{|1-\overline{a}z|^2}.
\end{equation*}

Observe that since $\rho(a, \lambda(a))\leq\gamma$, there is a constant $C_2=C_2(\gamma)>0$ such that

\begin{equation*}
    \int_{I(\lambda(a))}\frac{1-|z|^2}{|\xi-z|^2}\,|\mathrm{d}\xi|\geq C_2\hspace{1pt}\frac{(1-|a|^2)(1-|z|^2)}{|1-\overline{a}z|^2}.
\end{equation*}
Here $I(\lambda(a)) = \{ \xi \in \partial \mathbb{D} : |\xi - \lambda(a) / |\lambda(a)| | \leq 1- |\lambda(a)| \}$ is the Privalov shadow of $\lambda(a)$. We add up these inequalities and, since any $\lambda\in\Lambda$ will be repeated at most $k$ times, we get the result.
\end{proof}

A sequence $A= \{a_k \}$ of points in $\mathbb{D}$ is called separated if $\eta (A) = \inf \{ \rho(a_1 ,  a_2 ) : a_1, a_2 \in A , a_1 \neq a_2  \} >0$. The number $\eta(A)$ is called the separation constant of $A$.  

\begin{corollary}
\label{corolari1}
Let $\Lambda$ be a Blaschke sequence and let $A$ be a separated sequence of points in $\mathbb{D}$ with separation constant $\eta= \eta (A) $. Assume that there exists $0<\gamma<1$ such that for any $a\in A$ there is $\lambda(a)\in\Lambda$ with $\rho(a,\lambda(a))\leq\gamma$. Then, $A$ is a Blaschke sequence and for any $0<\delta<1$, there is a constant $C=C(\eta,\gamma,\delta)>0$ such that

\begin{equation*}
     \mathlarger{\mathlarger{\sum_{a\in A:\rho(a,z)>\delta}}}\log\left|\frac{a-z}{1-\overline{a}z}\right|^{-1} \leq CH_\Lambda(z),\hspace{0.5cm}z\in\mathbb{D}.
\end{equation*}
\end{corollary}

\begin{proof}
Since $A$ is a separated sequence, there exists a constant $k=k(\gamma,\eta)>0$ such that $\#\{a\in A:\lambda(a)=\lambda\}\leq k$. Then, the result follows from Lemma \ref{lema2}.
\end{proof}
We are now ready to prove Theorem \ref{teorema1}. 

\begin{proof1}
     (a) We first show that there exists a universal constant $C>0$ such that for any $f\in H^\infty$, $||f||_\infty\leq1$ satisfying (\ref{parta}), there exists $H_1=H_1(f, B)\in$ Har$_+($\D$)$ such that
        
        \begin{equation}
        \label{eq3}
            -\log(|B(z)|+|f(z)|)\leq H_1(z),\hspace{0.5cm}z\in\mathbb{D}.
        \end{equation}
        
        By Lemma \ref{lema1}, one only needs to check (\ref{eq3}) for points $z\in\mathbb{D}$ with $\rho(z,\Lambda)\leq\frac{1}{2}$. Therefore, fix $z\in\mathbb{D}$ and $\lambda_k\in\Lambda$ with $\rho(z,\lambda_k)\leq\frac{1}{2}$. Assume that $\rho(z,\lambda_k)\leq e^{-H(z)}$. By the Schwarz Lemma, $\rho(f(z),f(\lambda_k))\leq e^{-H(z)}$. Now, the inequalities 
        
        \begin{equation*}
            e^{-H(z)}\geq\rho(f(z),f(\lambda_k))\geq\frac{|f(\lambda_k)|-|f(z)|}{1-|f(z)||f(\lambda_k)|}
        \end{equation*}
        
        yield
        
        \begin{equation*}
            |f(z)|\geq\frac{|f(\lambda_k)|-e^{-H(z)}}{1-e^{-H(z)}|f(\lambda_k)|}.
        \end{equation*}
 Let $C>0$ be a sufficiently small universal constant to be fixed later.  Then, the assumption (\ref{parta}) gives
        
          \begin{equation*}
            |f(z)|\geq\frac{e^{-CH(\lambda_k)}-e^{-H(z)}}{1-e^{-H(z)}e^{-CH(\lambda_k)}}.
        \end{equation*}
        
        By Harnack's inequality, $H(z)$ and $H(\lambda_k)$ are comparable. Hence we can choose a sufficiently small universal constant $C>0$ such that 
        \begin{equation*}
            |f(z)|\geq\frac{1}{2}e^{-CH(z)}.
        \end{equation*}
        
        Hence, $-\log|f|$ has a harmonic majorant on the set $\{z\in\mathbb{D}:\rho(z,\Lambda)\leq e^{-H(z)}\}$. By assumption, $- \log|B|$ has a harmonic majorant on the set $\{z\in\mathbb{D}:\rho(z,\Lambda)\geq e^{-H(z)}\}$ and (\ref{eq3}) follows. Now, apply the Corona Theorem for the Nevanlinna class to obtain functions $g,h\in\mathcal{N}$ such that $fg=1+Bh$.
        
        (b) Consider now the set $\Tilde{\Lambda}=\left\{z\in\mathbb{D}:\rho(z,\Lambda)\leq 1/ 2 \right\}$ and the family $\mathcal{F}$ of dyadic Whitney squares $Q$ such that $Q\cap\Tilde{\Lambda}\neq\emptyset$. Consider the set $E=\{z\in\mathbb{D}:\rho(z,\Lambda)< e^{-H(z)}\}$. Observe that there is a universal constant $C_1 >0$ such that $H_\Lambda (z) > C_1 N(Q)$ for any $z \in Q$ and any dyadic Whitney square $Q$. Hence $H(z) \geq C_0 C_1 N(Q)$ for any $z \in Q$ and if the constant $C_0 $ is taken sufficiently large, we deduce that $\overline{Q} \setminus E \neq \emptyset$.  For each $Q\in\mathcal{F}$, pick $a(Q)\in\overline{Q}\backslash E$ such that  
        \begin{equation*}
            \log|B(a(Q))|^{-1}=\max\{\log|B(z)|^{-1}:z\in\overline{Q}\backslash E\}.
        \end{equation*}
        
        Consider the sequence $A=\{a(Q):Q\in\mathcal{F}\}$ and pick a positive number $\delta$ 
such that for any $z\in \DD$, the hyperbolic disc $D(z,\delta)$ intersects no more
than four distinct dyadic Whitney squares.
        Since $A$ is the union of at most four $\delta$-separated sequences, Corollary \ref{corolari1} gives that $A$ is a Blaschke sequence. Moreover, if $B_A$ denotes the Blaschke product with zero set $A$, there exists a constant $C_2=C_2 (\delta)>0$ such that   
        \begin{equation}
            \label{eq4}
            \log|B_A (z)|^{-1}\leq C_2 H_{\Lambda} (z)+4\log\rho(z,A)^{-1},\hspace{0.5cm}z\in\mathbb{D}.
        \end{equation}
        
        Observe that since $\delta>0$ is a universal constant, so is $C_2$. Fix $\lambda_k\in\Lambda$. Let $Q_k$ be the Whitney square containing $\lambda_k$, and $a_k\in A$ such that $\rho(\lambda_k,a_k)=\rho(\lambda_k,A)$. Hence, either $a_k\in Q_k$ or $a_k$ belongs to a Whitney square $Q$ with $\overline{Q}\cap\overline{Q_k}\neq\emptyset$. Let $H\in$ Har$_+($\D$)$, $H\geq C_0 H_\Lambda$. Harnack's inequality gives that there exists a universal constant $C_3>0$ such that $C_3 H(\lambda_k)\geq H(a_k)$. Since $\rho(a_k,\lambda_k)\geq e^{-H(a_k)}$, we deduce that $\rho(a_k,\lambda_k)\geq e^{-C_3H(\lambda_k)}$. Using (\ref{eq4}) we deduce that
        \begin{equation*}
            \log|B_A(\lambda_k)|^{-1}\leq C_2 H_\Lambda(\lambda_k)+4 C_3 H(\lambda_k).
        \end{equation*}
        By assumption, $H\geq C_0 H_\Lambda$ and we deduce that
        
        \begin{equation*}
            \log|B_A(\lambda_k)|^{-1}\leq (C_2 C_0^{-1} + 4 C_3 )H(\lambda_k),
        \end{equation*}
that is, the function $f=B_A$ satisfies estimate \eqref{parta} with the constant $C=C_2 C_0^{-1} + 4C_3$. By assumption, there exist $g,h\in\mathcal{N}$ such that $Bg+B_Ah=1$. We deduce that there exists $H_1\in$ Har$_+($\D$)$ such that $-\log|B(a_k)|\leq H_1(a_k)$, $k=1, 2, \ldots$. Hence, $-\log|B|$ has a harmonic majorant in $\Omega=\bigcup_{Q\in\mathcal{F}} (Q\backslash E)$. Since, by Lemma \ref{lema1}, $-\log|B|$ has a harmonic majorant on $\mathbb{D}\backslash(\bigcup_{Q\in\mathcal{F}} Q)$, the proof is complete.
    
\end{proof1}

The proof of part (a) can be easily adapted to show the following more general fact:
\begin{corollary}
\label{corona}
Let $B$ be a Blaschke product with zero set $\Lambda = \{\lambda_k \}$. Then there exists a universal constant $C>0$ such that the following statement holds. Let $H\in$ Har$_+($\D$)$ and assume that the function $-\log|B|$ has a harmonic majorant on the set $\{z\in\mathbb{D}:\rho(z,\Lambda)\geq e^{-H(z)}\}$. Then for any $f_1 , \ldots , f_n \in H^\infty$, $||f_i||_\infty\leq1$, $i=1, \ldots, n$, such that
            \begin{equation*}
             \sum_{i=1}^n |f_i (\lambda_k)|> e^{-CH(\lambda_k)}, \quad k=1,2, \ldots , 
            \end{equation*}
        there exist $g_1, \ldots , g_n,h\in\mathcal{N}$ such that $\sum_i f_ig_i=1+Bh$. 
\end{corollary}
On the other hand, the $n$-tuple analogue of part (b) trivially holds, because the hypothesis for $n\ge 1$ implies the hypothesis for $n=1$ which is the one used in Theorem \ref{teorema1}.

\vskip.3cm
The proof of Theorem \ref{teorema2} uses the following auxiliary result.

\begin{lemma}
\label{lema3}
Let $\{Q_j\}$ be an infinite sequence of different dyadic Whitney squares and let $\{M_j\}$ be a sequence of positive numbers with $\lim_{j \to \infty} M_j=\infty$. Then there exists $H\in$ Har$_+($\D$)$ 
and a constant $C_0 >0$ 
such that $H_j=\sup\{H(z):z\in Q_j\}$ satisfies $H_j\leq M_j+C_0$ for any $j=1, 2, \ldots$, 
and $\limsup_{j \to \infty} H_j=\infty$.
\end{lemma}

\begin{proof}

For any Whitney cube $Q$, let us set
\begin{equation}
\label{hq}
h_Q(z):= \int_{I(Q)} \frac{1-|z|^2}{|e^{i\theta}-z|^2} \frac{d\theta}{2\pi},
\end{equation}
where $I(Q)$ 
is the radial projection of $Q$ onto $\partial \mathbb{D}$.

Note that there exists an absolute constant $c>0$ such that the function $h_Q$ verifies the following properties:
\begin{equation}
\label{prophq}
0\le h_Q (z) \le 1, \forall z \in \DD; \quad h_Q (z) \leq \frac{l(Q)}{c(1-|z|)};
\quad 0 < c \le  h_Q (z), \forall z \in Q.
\end{equation}

We will construct inductively a sequence of coefficients $\mu_m$ and an increasing sequence of integers
$(j_m)$
so that $H:=\sum_m \mu_m h_{Q_{j_m}}$
satisfies the conclusion of the Lemma.  There is no loss of generality in assuming
that $l(Q_{j+1})\le l(Q_{j})$ for all $j$.  Let $H_0:=0$. For any $k >0$, let us denote $H^{(k)}:=\sum_{m=1}^k \mu_m h_{Q_{j_m}}$.  Let $H^{(k)}_j:=\sup_{Q_j} H^{(k)}$. We want to prove by induction on $k$  that:
\begin{equation}
\label{recdom}
 H^{(k)}_j \le M_j + \sum_{m=1}^k 2^{-m/2}, j=1,2,\ldots . 
\end{equation}
In particular, this will show that $\sup_k H^{(k)}$ will be bounded on any square $Q_j$,
therefore $H:= \sum_{j=1}^\infty \mu_m h_{Q_{j_m}}$ will be well defined and will  satisfy $H_j\le M_j +C_0$,
for all $j \in \Z_+$.

For $k=0$, the property \eqref{recdom} is vacuously true.  Suppose it is verified for 
$k$. Since $H^{(k)}$ is bounded and $M_j \to \infty$,
there exists $R \in (0, l(Q_{j_k}) )$ such that for any $j \in \Z_+$ such that
$l(Q_j) \le R$, then $M_j - H^{(k)}_j \ge 2^{(k+1)/2}$. Now there exists $R'<R$ such that for all $z$ such that $|z| \le 1-R/2$ (in particular,
for $z\in Q_j$ with $j\le j_k$) and for all $Q$ such that $l(Q) \le R'$, 
$h_Q(z) < 2^{-k-1}$. Pick $j_{k+1}$ to be the smallest $j$ such that $l(Q_j) \le R'$, and $\mu_{k+1}:=2^{(k+1)/2}$.
Then for all $j$ such that $l(Q_j)\ge R$, by the induction hypothesis,
\[
H^{(k+1)}_j = H^{(k)}_j +  \mu_{k+1} h_{Q_{j_{k+1}}} 
\le H^{(k)}_j + 2^{(k+1)/2} 2^{-k-1} \le M_j +  \sum_{m=1}^{k+1} 2^{-m/2}.
\]
On the other hand, for all $j$ such that $l(Q_j)\le R$, 
\[
H^{(k+1)}_j = H^{(k)}_j +  \mu_{k+1} h_{Q_{j_{k+1}}} \le H^{(k)}_j +  \mu_{k+1}
= H^{(k)}_j + 2^{(k+1)/2} \le M_j,
\]
by the choice of $R$. The inductive condition \eqref{recdom} is verified. Finally, notice that for $z \in Q_j$ so that $j=j_m$ for some $m$, 
\[
H(z) \ge \mu_m h_{Q_{j_m}} \ge c 2^{m/2} \to \infty \mbox{ as } m\to \infty.
\]
\end{proof}

\begin{proof2}
(a) For each dyadic Whitney square $Q$ let $N(Q)=\#(Q\cap\Lambda)$ and let $\mathcal{U}(Q)$ be the collection of at most nine dyadic Whitney squares $Q_1$ such that $\overline{Q_1}\cap\overline{Q}\neq\emptyset$. Observe that there exists an absolute constant $\delta>0$ such that
     \begin{equation*}
         \delta\;\leq\;\rho(Q,\mathbb{D}\;\backslash\;\mathcal{U}(Q))
         := \inf \{\rho (z, w): z \in Q , w \in \mathbb{D}\;\backslash\;\mathcal{U}(Q) \}
     \end{equation*}
for any dyadic Whitney square $Q$. 
Consider also $M(Q)=\#(\mathcal{U}(Q)\cap\Lambda)$. Let $\{Q_j\}$ be the collection of dyadic Whitney squares such that $M(Q_j)>0$. The Blaschke condition gives that
     \begin{equation*}
         \mathlarger{\sum}M(Q_j)l(Q_j)<\infty.
     \end{equation*}
Then there exists a sequence $\{\Tilde{M_j}\}$, $\Tilde{M_j} \geq M(Q_j)$ for any $j \geq 1$, with 
     \begin{equation*}
         \lim_{j\to\infty} \Tilde{M_j} \big/ M(Q_j)=+\infty \text {   and   } \sum\Tilde{M_j} l(Q_j)<\infty.
     \end{equation*}
Lemma \ref{lema3} provides $H\in$ Har$_+($\D$)$ such that $H(z) \leq \Tilde{M_j}\big/ M(Q_j) +C_0$ for any $z\in Q_j$ and $\limsup_{j\to\infty}\sup\{H(z):z\in Q_j\}=+\infty$. Since the sequence $\Lambda$ is contained in $\cup Q_j $, Harnack's inequality gives that $\limsup_{k\to\infty} H(\lambda_k)=+\infty$. We will now show that the function $-\log|B|$ has a harmonic majorant on the set $\{z\in\mathbb{D}:\rho(z,\Lambda)\geq e^{-H(z)}\}$. Since $\rho(\Lambda,\mathbb{D}\backslash \cup Q_j)\geq\delta>0$, Lemma \ref{lema1} gives that $-\log|B|$ has a harmonic majorant on $\mathbb{D}\backslash \cup Q_j$. Now fix $z\in Q_j$ with $\rho(z,\Lambda)\geq e^{-H(z)}$ and split $\Lambda=\Lambda_1\cup\Lambda_2$, where $\Lambda_1=\{\lambda_k:\rho(\lambda_k,z)\leq\delta\}$ and $\Lambda_2=\{\lambda_k:\rho(\lambda_k,z)>\delta\}$. By Lemma \ref{lema1}, there exists a constant $C=C(\delta)>0$ such that

\begin{equation*}
    \mathlarger{\sum}_{\lambda_k\in\Lambda_2}\log\rho(\lambda_k,z)^{-1}\leq CH_\Lambda(z).
\end{equation*}
On the other hand, since $\rho(\lambda_k,z)\geq e^{-H(z)}$, we have

\begin{equation*}
    \mathlarger{\sum}_{\lambda_k\in\Lambda_1} \log\rho(\lambda_k,z)^{-1}\leq H(z)M(Q_j) \leq \Tilde{M_j} + C_0 M(Q_j) \leq (1+ C_0) \Tilde{M_j} .
\end{equation*}

Consider the harmonic function $H_1:=\sum_j\Tilde{M_j} h_{Q_j}$, where $h_Q$ is as in \eqref{hq}.
By the last estimate in \eqref{prophq}, 
$H_1(z)\geq c\Tilde{M_j}$ for any $z\in Q_j$. We deduce that 
\begin{equation*}
    \log|B(z)|^{-1}\leq CH_\Lambda(z)+ c^{-1} (1+ C_0) H_1(z),
\end{equation*}
and this finishes the proof of part (a).

To prove part (b), let $C$ be as in Theorem \ref{teorema1} (a), and 
$\tilde H$ a function as in part (a). If we set $H:=C\tilde H$,
applying Theorem \ref{teorema1} (a) yields our result.

\end{proof2}

\vspace{0.5cm}
\textbf{A Family of Blaschke products}

We now present a family of Blaschke products $B$ for which the family $\mathcal{H}(B)$ can be easily described. Let $\Lambda=\{\lambda_k\}$ be a separated sequence in $\mathbb{D}$, that is, assume  $\eta=\inf\{\rho(\lambda_k,\lambda_j):k\neq j\}>0$. Let $N =\{N_j\}$ be a sequence of positive integers tending to infinity such that
\begin{equation*}
    \sum N_j(1-|\lambda_j|)<\infty.
\end{equation*}
Consider the Blaschke product $B(\Lambda, N)$ defined as
\begin{equation}
\label{def}
    B(\Lambda,N)(z)=\prod_j\left(\frac{\overline{\lambda_j}}{|\lambda_j|}\frac{\lambda_j-z}{1-\overline{\lambda_j}z}\right)^{N_j},\hspace{0.5cm}z\in\mathbb{D}.
\end{equation}
Consider the pairwise disjoint pseudohyperbolic disks $D_j=\{z\in\mathbb{D}:\rho(z,\lambda_j)\leq \eta / 4\}$, $j=1, 2, \ldots$. By Lemma \ref{lema1}, $-\log|B(\Lambda, N)|$ has a harmonic majorant on $\mathbb{D}\backslash\cup_j D_j$. Again, by Lemma \ref{lema1}, there exists a constant $C>0$ and a function $H_1 \in $ Har$_+(\mathbb{D})$ such that
\begin{equation*}
    \mathlarger{\sum}_{k\neq j}\log\rho(z,a_k)^{-N_k}\leq CH_1 (z),\hspace{0.5cm}z\in D_j,\hspace{5pt}j=1, 2, \ldots.
\end{equation*}
Fix $H\in$ Har$_+(\mathbb{D})$. Then, $-\log|B(\Lambda, N)|$ has a harmonic majorant on $\{z\in\mathbb{D}:\rho(z,\Lambda)\geq e^{-H(z)}\}$ if and only if there exists $H_1\in$ Har$_+(\mathbb{D})$ such that
\begin{equation*}
    N_j\log\rho(z,a_j)^{-1}\leq H_1(a_j),\hspace{0.5cm}j=1, 2, \ldots, 
\end{equation*}
whenever $\rho(z,a_j)\geq e^{-H(z)}$. Hence, $-\log|B(\Lambda, N)|$ has a harmonic majorant on $\{z\in\mathbb{D}:\rho(z,\Lambda)\geq e^{-H(z)}\}$ if and only if the mapping $a_j\rightarrow N_jH(a_j)$ has a harmonic majorant. Hence for the Blaschke products $B= B(\Lambda , N)$ we have $H \in \mathcal{H} (B)$ if and only $H$ satisfies the sufficient condition given in Theorem \ref{teorema4}.

\section{Theorems 3,4 and 5}
We start with the proof of part (a) of Theorem \ref{teorema5}. 
\begin{proof}[Proof of Theorem \ref{teorema5} (a)]
Let $a_j\in\mathbb{D}$ with 
\begin{equation*}
    \lim_{j\to\infty}\frac{H_1(a_j)}{H_2(a_j)} = \infty.
\end{equation*}
Considering a subsequence if necessary, we can assume that $\rho(a_j,a_k)\geq 1/2$ if $k\neq j$. By Harnack's inequality $H_1 (a_j) (1 - |a_j|) \leq 2 H_1 (0)$ for any $j$. Hence $\lim_{j \to \infty} (1-|a_j|) H_2 (a_j) =0$. Pick a sequence $\{N_j\}$ of positive integers such that $\lim_{j \to \infty} N_j(1-|a_j|)H_2(a_j)=0$ and $\lim_{j \to \infty} N_j(1-|a_j|)H_1(a_j)=+\infty$. Considering again a subsequence of $\{a_j\}$ if necessary, we can assume that
\begin{equation}
\label{sumN}
    \sum N_j(1-|a_j|)H_2(a_j)<\infty.
\end{equation}
Now let $B$ be the Blaschke product defined by
\begin{equation*}
    B(z)=\mathlarger{\prod_j} \frac{\overline{a_j}}{|a_j|}\left(\frac{a_j-z}{1-\overline{a_j}z}\right)^{N_j},\hspace{0.5cm} z\in\mathbb{D}.
\end{equation*}

As discussed at the end of the previous section, for any $H\in$ Har$_+(\mathbb{D})$, the function $-\log|B|$ has a harmonic majorant on the set $\{z\in\mathbb{D}:\rho(z,\{a_j\})\geq e^{-H(z)}\}$ if and only if the mapping $F(H)$ defined by  $F(H)(a_j)=N_jH(a_j)$, $j\geq 1$, and $F(H)(z)=0$ if $z \notin \{a_j \}$, has a harmonic majorant. Since $\lim_{j \to \infty} N_j H_1(a_j)(1-|a_j|)=+\infty$, the mapping $F(H_1)$ can not have a harmonic majorant. Consider the function 
\begin{equation*}
    H_3(z)= \mathlarger{\sum_j} N_j H_2(a_j)h_{Q_j},
\end{equation*}
where 
$Q_j$ is the dyadic Whitney square containing $a_j$. Here $h_Q$ is the function defined in \eqref{hq}. Since $l(Q_j)$ is comparable to $ 1-|a_j|$, the above
sum converges by \eqref{sumN}. 
Observe that last estimate of \eqref{prophq} gives that  there exists an absolute constant $C_1>0$ such that
\begin{equation*}
    H_3(a_j)\geq C_1N_jH_2(a_j),\hspace{0.5cm} j=1, 2, \ldots
\end{equation*}
Hence $F(H_2)$ has a harmonic majorant.
\end{proof}

\begin{proof}[Proof of Theorem \ref{teorema5} (b)]
By Harnack's inequality, there is a constant $\gamma\in (0,1)$ such that for any dyadic 
Whitney square $Q$, any positive harmonic function $H$, any $z,z'\in Q$, we have 
$\gamma H(z') \le H(z) \le \gamma^{-1} H(z')$.  Pick $\eta \in (0,\eta_0)$.

Given an unbounded positive harmonic function  $H$, we can choose a 
sequence of dyadic Whitney squares $\{Q_j\}$ such that 
\begin{enumerate}
\item
$l(Q_{j+1}) < l(Q_{j})$, $j\in \Z_+$, and
\item 
If $z_j$ denotes the center of $Q_j$, 
\[
H(z_{j}) \to \infty \mbox{ and } H(z_j) \le \frac{\gamma}{2(1+\eta)} \log \frac1{l(Q_j)}.
\]
\end{enumerate}
To prove that we can satisfy the second condition, in the case where 
$$
\max_{z:|z|=1-l(Q_j)} H(z) \le 
 \frac{\gamma}{2(1+\eta)} \log \frac1{l(Q_j)}, 
 $$
it is enough to choose $Q_j$ to be the dyadic Whitney  square where the maximum is attained; otherwise, we know that the average value of $H$ on the circle $\{z \in \mathbb{C} : |z|=1-l(Q_j)\}$ is $H(0)$,
 so for $j$ large enough we can find a dyadic square $Q_j$ where 
\begin{equation}
\label{gamma}
\frac{\gamma^3}{2(1+\eta)}  \log \frac1{l(Q_j)} \le H(z_j) \le  \frac{\gamma^2}{2(1+\eta)}\log \frac1{l(Q_j)} . 
\end{equation}
Observe that since $H$ is unbounded we have $\lim_{j \to \infty} H(z_j) = \infty$. 
We shall need to take subsequences of $\{Q_j \}$, while keeping the same name for the sequence. Choose a sequence $R_j\to0$ such that \begin{equation}
\label{defR}
\lim_{j\to\infty} \frac{\log R_j^{-1}}{H(z_j)} =0 .
\end{equation}
Observe that $\lim_{j\to\infty} l(Q_j)^2 H(z_j) \log \frac1{R_j} =0$. Indeed, for $j$ large enough, we have
\[
0 < l(Q_j)^2 H(z_j) \log \frac1{R_j}  \le
 l(Q_j)^2 H(z_j)^2 \le \left( l(Q_j) \frac{\gamma^2}{2(1+\eta)} \log \frac1{l(Q_j)} \right)^2.
 \]
Now, with $[\cdot]$ denoting the integer part of a real number, define the sequence
of integers
\begin{equation}
\label{defN}
N_j:= \left[ 
\frac1{l(Q_j) \left(H(z_j) \log \frac1{R_j}\right)^{1/2}}
\right].
\end{equation}

For $z_0\in\mathbb{D}$ and $t>0$, let $D_\rho(z_0,t)=\{z\in\mathbb{D}:\rho(z,z_0)\leq t\}$ denote the pseudohyperbolic disk of radius $t$ centered at $z_0$. We define the sequence $\Lambda$ as the union of finite sequences $\Lambda^{(k)} \subset Q_k$.
For each $k$, $\Lambda^{(k)}$ is the union of 
\begin{enumerate}
\item
the point $z_k$ with multiplicity $N_k$ and
\item
a maximal subset of points $\lambda_j = \lambda_j (k)$ contained in the pseudohyperbolic disc $D_\rho (z_k, R_k)$
such that for any $i\neq j$, $ \rho(\lambda_i,\lambda_j) \ge e^{-(1+\eta)H(z_k)}$.  
\end{enumerate}
Note that we are adding to the multiple zero $z_k$ a set of points $\lambda_j$ with a cardinality on 
the order of $R_k^2 e^{2(1+\eta)H(z_k)}$,  which tends to infinity by \eqref{defR}. We proceed to take a subsequence of $\Lambda$ (still denoted by the same letter)
that will make it, among other things, a Blaschke sequence.

First observe that 
\begin{equation}
\label{vanishlog}
\lim_{j\to\infty} l(Q_j) N_j \log \frac1{R_j} 
= \lim_{j\to\infty}  \left( \frac{\log \frac1{R_j}}{H(z_j)} \right)^{1/2} =0,
\end{equation}
and 
\begin{equation}
\label{blowupH}
\lim_{j\to\infty} l(Q_j) N_j H(z_j) = \lim_{j\to\infty}  \left( \frac{H(z_j)}{ \log \frac1{R_j} } \right)^{1/2}
= \infty.
\end{equation}
On the other hand, applying the second inequality in \eqref{gamma} and \eqref{defR}, one gets 
\begin{equation}
\label{vanishexp}
 \lim_{j\to\infty}  \left(  l(Q_j)  e^{2(1+\eta)H(z_j)} R_j^2 \log \frac1{R_j} \right)=0.
\end{equation}
We now complete the definition of $\Lambda$ by restricting the indexes $k$ such that $\rho(Q_k , Q_j) \geq 1/2$ if $k \neq j$ and taking a subsequence so that 
\begin{equation}
\label{summable}
\sum_{j=1}^\infty l(Q_j) \left(N_j \log \frac1{R_j}+    e^{2(1+\eta)H(z_j)} R_j^2 \log \frac1{R_j} \right) < \infty,
\end{equation}
which is possible by \eqref{vanishlog} and \eqref{vanishexp}. Let $\Lambda= \cup \Lambda^{(k)}$ be the resulting sequence. 

Observe that an immediate consequence of this is that $\Lambda$ is now a Blaschke sequence, since
\[
\sum_{j=1}^\infty l(Q_j) \left(N_j +    e^{2(1+\eta)H(z_j)} R_j^2 \right) < \infty.
\]

{\bf Claim 1.}
For $k$ large enough, 
\[
\{ \zeta \in Q_k: \rho(\zeta, \Lambda) \ge e^{-H(\zeta)} \} \cap D_\rho (z_k,R_k)
=\emptyset.
\]
\begin{proof}
For any $\zeta \in D(z_k,R_k)$, there is a $\lambda_j$ such that 
$\rho(\zeta,\lambda_j)<  e^{-(1+\eta)H(z_k)}$.
Since $\lim_{k \to \infty} R_k =0$, by Harnack's inequality, there is a number $\gamma_k <1$ with $\lim_{k\to \infty} \gamma_k=1$, such that $\gamma_k H(z_k) \le H(z) \le \gamma_k^{-1} H(z_k)$  for any $z\in D(z_k,R_k)$. Then, for $k$ large enough,
\[
\log \rho(\zeta,\lambda_j) < -(1+\eta)H(z_k) 
\le - \gamma_k^{-1} H(z_k) \le - H(\zeta).
\]
\end{proof}
We henceforth restrict attention to the tail of the sequence where the conclusion of Claim 1 holds.
\vskip.3cm

{\bf Claim 2.}
$H \in \mathcal H(B)$.
\begin{proof}
Since $\rho(Q_k , Q_j ) \geq 1/2$ if $k \neq j$, it is enough to majorize, on each
$Q_k$, the part of the product corresponding to the local zeros, that is, to find $H_1 \in $ Har$_+(\mathbb{D})$ such that 
\[
\sum_{\lambda \in \Lambda \cap Q_k} \log\frac1{\rho(\zeta,\lambda)} \leq H_1 (z_k), \quad  \zeta \in Q_k, k =1,2,\ldots, 
\]
if $\rho (\zeta, \Lambda) \geq e^{- H (\zeta)}$. By the previous Claim this only occurs when $\zeta \notin D_\rho (z_k , R_k)$.
Then the above sum breaks into two terms: those corresponding
to $\lambda = z_k$ can be estimated by  $- N_k \log R_k$, and those admit a harmonic majorant 
because by \eqref{summable}, 
$$\sum_{j=1}^\infty l(Q_j) N_j \log \frac1{R_j} <\infty . $$  
The second term corresponds to the points $\lambda = \lambda_j \in \Lambda^{(k)} \setminus \{z_k \}$. After applying an automorphism of the disc mapping $\lambda_k$ to $0$, the corresponding sum 
$$
\sum_{\lambda_ j \in \Lambda^{(k)} \setminus \{z_k \} } \log\frac1{\rho(\zeta,\lambda)}
$$
reduces to a Riemann sum for the area integral of $\log \frac1{|z|}$, with
disks of (Euclidean) radius $e^{-(1+\eta)H(z_k)}$. The integral is convergent,
and after an elementary computation, one finds that the second term is bounded by a fixed multiple of  
$e^{2(1+\eta)H(z_k)} R_k^2 \log \frac1{R_k} $. Again by \eqref{summable}, this term also admits a harmonic majorant. 
\end{proof}

We now want to show that $-\log |B|$
has no harmonic majorant on the set $\{z: \rho(z,\Lambda)> e^{-(1+\eta_0)H(z)}\}$.

{\bf Claim 3.} For $k$ large enough, 
\[
\{ \zeta \in Q_k: \rho(\zeta, \Lambda) \ge e^{-(1+\eta_0 )H(\zeta)} \} \cap D(z_k,e^{-(1+\eta)H(z_k)})
\neq \emptyset.
\]
\begin{proof}
The choice of the points $\{\lambda_j \}$ gives that there is a point $\zeta$ such that $\rho(z_k,\zeta)= e^{-(1+\eta)H(z_k)}$ and for any $\lambda_j \in \Lambda^{(k)} \setminus \{z_k \}$, we have 
$\rho(\zeta,\lambda_j)\ge \frac12 e^{-(1+\eta)H(z_k)}$. 
Therefore, since $\eta < \eta_0$,  for $k$ large enough,
\begin{multline*}
\log \rho(\zeta, \Lambda) \ge -\log 2 -(1+\eta)H(z_k)
\ge -\log 2  -\gamma_k^{-1}(1+\eta)H(\zeta) 
\\
\ge -(1+\eta_0)H(\zeta).
\end{multline*}
\end{proof}

Let $\zeta$ be a point in the non empty intersection given by Claim 3. Since $B$ has a zero at $z_k$ of multiplicity $N_k$,  , 
\[
\log \frac1{|B(\zeta)|} \ge N_k (1+\eta )H(z_k)
\]
and \eqref{blowupH} implies that this cannot admit a harmonic majorant,
because any majorizing function would have to grow faster than $1/l(Q_k)$ 
at the points $z_k$.
\end{proof}
\vskip.5cm

The proof of Theorem \ref{teorema3} uses the following variant of Lemma 1.1 of \cite{MNT}. 

\begin{lemma}
\label{lema4}
There exists a universal constant $C_0\geq1$ such that the following statement holds. Let $\Lambda$ be a Blaschke sequence and $H\in$ Har$_+(\mathbb{D})$. Let $z\in\mathbb{D}$ with $e^{H(z)}\geq\max\{C_0,\#\{\lambda\in\Lambda:\rho(\lambda,z)\leq\frac{1}{2}\}\}$. Then there exists $\Tilde{z}\in\mathbb{D}$ with $\rho(\Tilde{z},\Lambda)\geq e^{ - H (\Tilde{z} )}$ and $\rho(\Tilde{z},z)\leq e^{-H(z)/C_0}$.
\end{lemma}

\begin{proof}
We can assume that $H(z)\geq100$. A calculation shows that there exists a constant $C_1>1$ such that
\begin{equation*}
    C_1^{-1}t^2(1-|z|)^2\leq \text{ Area }D_\rho(z,t)\leq C_1t^2(1-|z|)^2.
\end{equation*}
Using these estimates and the fact that $H(z)\geq100$, one can show that there exists a sufficiently large universal constant $C_0>0$ such that the pseudohyperbolic disk $D_\rho(z,e^{-H(z)/C_0})$ contains more than $e^{3H(z)/2}$ pairwise disjoint pseudohyperbolic disks $D_j$ of pseudohyperbolic radius $e^{-H(z_j)}$. Here $z_j$ denotes the center of $D_j$. Since $e^{H(z)}\geq\#(\Lambda\cap D_\rho(z,\frac{1}{2}))$, there exists at least one $D_j$ with $D_j\cap\Lambda=\emptyset$ and we can take as $\Tilde{z}$ the center of $D_j$.
\end{proof}

\begin{proof3}
Let $C_0\geq1$ be the constant appearing in Lemma \ref{lema4}. We can assume $C>1$. We will show that there exists a constant $C_1=C_1 (C)>0$ such that, for any $z\in\mathbb{D}$ with $C_0^{-1}\geq\rho(z,\Lambda)\geq e^{-CH(z)}$ there exists $\Tilde{z}\in\mathbb{D}$ with $\rho(\Tilde{z},\Lambda)\geq e^{-H(\Tilde{z})}$ and 
\begin{equation}
    \label{eq5}
    \log|B(z)|^{-1}\leq C_1 (\log|B(\Tilde{z})|^{-1}+H_\Lambda(\Tilde{z})).
\end{equation}

Fix $z\in\mathbb{D}$ with $C_0^{-1}\geq\rho(z,\Lambda)\geq e^{-CH(z)}$.  Apply Lemma \ref{lema4} to find $\Tilde{z}\in\mathbb{D}$ with $\rho(\Tilde{z},z)\leq e^{-H(z)/C_0}$ such that $\rho(\Tilde{z},\Lambda)\geq e^{-H(\Tilde{z})}$. Let $\Lambda = \{\lambda_k\}$ and split $\Lambda=\Lambda_1\cup\Lambda_2\cup\Lambda_3$ where $\Lambda_1=\{\lambda_k:\rho(z,\lambda_k)\leq e^{-H(z)/2C_0}\}$, $\Lambda_2=\{\lambda_k: e^{-H(z)/2C_0}<\rho(z,\lambda_k)\leq 1/2 \}$ and $\Lambda_3=\{\lambda_k:\rho(z,\lambda_k ) \geq 1/2 \}$. By lemma \ref{lema1}, there exists a constant $C_2>0$ such that
\begin{equation*}
    \mathlarger{\sum}_{\lambda_k\in\Lambda_3}\log\rho(z,\lambda_k)^{-1}\leq C_2 H_\Lambda(z).
\end{equation*}
If $\lambda_k\in\Lambda_2$, we have $\rho(\Tilde{z},\lambda_k)\leq\rho(z,\Tilde{z})+\rho(z,\lambda_k)\leq 2\rho(z,\lambda_k)$. Using the obvious estimate $2x\leq x^{1/2}$, which holds for $0\leq x\leq 1/2$, we deduce $\rho(\Tilde{z},\lambda_k)\leq\rho(z,\lambda_k)^{1/2}$. Hence,
\begin{equation*}
    \mathlarger{\sum}_{\lambda_k\in\Lambda_2}\log\rho(z,\lambda_k)^{-1}\leq 2\mathlarger{\sum}_{\lambda_k\in\Lambda_2}\log\rho(\Tilde{z},\lambda_k)^{-1}\leq 2\log|B(\Tilde{z})|^{-1}.
\end{equation*}
Finally, since $\rho(z,\Lambda)\geq e^{-CH(z)}$ we have that 
\begin{equation*}
    \mathlarger{\sum}_{\lambda_k\in\Lambda_1}\log\rho(z,\lambda_k)^{-1}\leq CH(z)\#\Lambda_1.
\end{equation*}

Observe that if $\lambda_k\in\Lambda_1$, then $\rho(\Tilde{z},\lambda_k)\leq\rho(z,\Tilde{z})+\rho(z,\lambda_k)\leq2e^{-H(z)/2C_0}$ and we deduce that there exists a universal constant $C_3 >0$ such that
\begin{equation*}
    \log|B(\Tilde{z})|^{-1}\geq\mathlarger{\sum}_{\lambda_k\in\Lambda_1}\log\rho(\Tilde{z},\lambda_k)^{-1}\geq C_3 \frac{H(z)}{C_0}\#\Lambda_1.
\end{equation*}
Hence, there exists a constant $C_4>0$ such that
\begin{equation*}
    \mathlarger{\sum}_{\lambda_k\in\Lambda_1}\log\rho(z,\lambda_k)^{-1}\leq C_4\log|B(\Tilde{z})|^{-1}.
\end{equation*}
Collecting these estimates one finds a constant $C_5 >0$ such that 
\begin{equation*}
    \log|B(z)|^{-1}\leq C_5 (\log|B(\Tilde{z})|^{-1}+H_\Lambda (z)).
\end{equation*}
Since by Harnack's inequality $H_\Lambda (z)$ and $H_\Lambda (\Tilde{z})$ are comparable, this proves (\ref{eq5}). Now (\ref{eq5}), the assumption and another application of Harnack's inequality,
 give that $-\log|B|$ has a harmonic majorant on the set $\{z\in\mathbb{D}: C_0^{-1} \geq\rho(z,\Lambda)\geq e^{-CH(z)}\}$. By Lemma \ref{lema1}, there exists a constant $C_6>0$ such that $-\log|B(z)|\leq C_6 H_\Lambda(z)$ if $\rho(z,\Lambda)\geq C_0^{-1}$. This completes the proof.
\end{proof3}
\vskip.3cm

\begin{proof4}
Fix $z\in\mathbb{D}$ with $\rho(z,\Lambda)\geq e^{-H(z)}$. Consider $\Lambda_1=\{\lambda_k:\rho(\lambda_k,z)\leq 1/2 \}$ and $\Lambda_2=\{\lambda_k:\rho(\lambda_k,z)> 1/2 \}$. By Lemma \ref{lema1}, there exists an absolute constant $C_1>0$ such that
\begin{equation*}
    \mathlarger{\sum}_{\lambda_k\in\Lambda_2}\log\rho(\lambda_k,z)^{-1}\leq C_1H_\Lambda(z).
\end{equation*}
On the other hand, since $\rho(z,\Lambda)\geq e^{-H(z)}$, we have
\begin{equation*}
    \mathlarger{\sum}_{\lambda_k\in\Lambda_1}\log\rho(\lambda_k,z)^{-1}\leq H(z)\#\Lambda_1.
\end{equation*}
Let $Q$ be the dyadic Whitney square containing $z$. Since there exists a universal constant $0<C_2<1$ such that each point $\lambda_k\in\Lambda_1$ satisfies $\rho(\lambda_k,z(Q))\leq C_2$, we deduce that there exists a constant $C_3>0$ such that $H(z)\#\Lambda_1 \leq C_3H_1(z)$. Hence, $C_1H_\Lambda+C_3H_1$ is a harmonic majorant of $-\log|B|$ on the set $\{z\in\mathbb{D}:\rho(z,\Lambda)\geq e^{-H(z)}\}$.
\end{proof4}
\vskip.5cm

\begin{corollary}
Let $B$ be a Blaschke product with zero set $\Lambda$. Let $H\in$ Har$_+(\mathbb{D})$ such that
\begin{equation*}
    \mathlarger{\sum}N(Q)H(z(Q))l(Q)<\infty,
\end{equation*}
where the sum is taken over all dyadic Whitney squares Q such that $N(Q)>0$. Then, $-\log|B|$ has a harmonic majorant on the set $\{z\in\mathbb{D}:\rho(z,\Lambda)\geq e^{-H(z)}\}$.
\end{corollary}

\begin{proofc5}
 Consider the harmonic function $H_1\in$ Har$_+(\mathbb{D})$ defined by 
 \begin{equation*}
     H_1(z)=\sum N(Q)H(z(Q)) h_Q, 
     \quad z \in \mathbb{D} , 
 \end{equation*}
where the sum is taken over all dyadic Whitney squares $Q$ with $N(Q)>0$. 
Observe that by \eqref{prophq}, there exists a positive constant $C>0$ such that $H_1(z(Q))\geq C N(Q)H(z(Q))$ for any $Q$ with $N(Q)>0$. Now the result follows from Theorem \ref{teorema4}.
\end{proofc5}

\section{Smirnov quotient algebras}

A quasi-bounded harmonic function is a harmonic function in the unit disc which is the Poisson integral of an integrable function in the unit circle.  We denote by $QB_+(\mathbb{D})$ the cone of positive quasibounded harmonic functions in $\mathbb{D}$.  An analytic function $f$ in $\mathbb{D}$ is in the Smirnov class $\mathcal{N}^+$ if the function 
$\log^+|f|$ has a quasi-bounded harmonic majorant in $\mathbb{D}$. 
A function in the Nevanlinna class is in the  Smirnov class 
if and only if its canonical inner-outer  factorization has no singular function in the denominator. Hence the Smirnov class $\mathcal{N}^+$ is
an algebra where the invertible  functions are exactly the outer functions. Interpolating sequences in $\mathcal{N}^+$ were described as those sequences $\{z_n \}$ of points in $\mathbb{D}$ for which  there exists $H \in QB_+(\mathbb{D})$ such that condition \eqref{NIS} holds. See Theorem 1.3 of \cite{HMNT}. Mortini's proved in  \cite[Satz 4]{M} the following Corona type Theorem in the Smirnov class. Given $f_1,\ldots,f_n\in \mathcal{N}^+$, the B\'ezout equation $f_1g_1+\cdots+f_ng_n\equiv1$ can be solved with functions $g_1,\ldots,g_n\in \mathcal{N}^+$ if and only if there exists $H\in QB_+(\mathbb{D})$ such that 
$$
\sum_{i=1}^n |f_i (z)| \geq e^{-H(z)}, \quad z \in \mathbb{D}.
$$

Given an inner function $I$ with zero set $\Lambda = \{\lambda_k \}$  we want to study invertibility in the quotient algebra $\mathcal{N}^+/I\mathcal{N}^+$. Let   $f\in\mathcal{N}^+$ and assume that 
the class $[f]$ is invertible in $\mathcal{N}^+ \big/ I\mathcal{N}^+$, that is, there exist $g,h \in \mathcal{N}^+ $ such that $fg=1 + Ih$. Then there exists $H\in QB_+(\mathbb{D})$ such that 
\begin{equation}
    |f(\lambda_k)|\geq e^{-H(\lambda_k)}, \hspace{0.5cm} k=1,2,\ldots.
    \label{eqnova}
\end{equation}

We are interested on studying the converse statement. Observe that if $I$ had a non constant singular inner factor, then for any $h \in \mathcal{N}^+$ there would exist $\xi \in \partial \mathbb{D}$ such that $I(z)h(z)$ would tend to zero when $z$ approaches $\xi$ non-tangentially. Hence if $I$ has a non constant singular inner factor, we can not expect that condition \eqref{eqnova} implies that $[f]$ is invertible in $\mathcal{N}^+ \big/ I\mathcal{N}^+$. When $I$ is a Blaschke product, we have the following analogue of Theorem 1. 

\begin{theorem}
\label{teorema1smirnov}
Let $B$ be a Blaschke product with zero set $\Lambda=\{\lambda_k\}$.
    \begin{enumerate}[(a)]
        \item There exists a universal constant $C>0$ such that the following statement holds. Let $H\in QB_+ (\mathbb{D})$ and assume that the function $-\log|B|$ has a quasibounded harmonic majorant on the set $\{z\in\mathbb{D}:\rho(z,\Lambda)\geq e^{-H(z)}\}$. Then for any $f\in H^\infty$, $||f||_\infty\leq1$ such that
            \begin{equation*}
             |f(\lambda_k)|> e^{-CH(\lambda_k)}, \quad k=1,2, \ldots , 
            \end{equation*}
        there exist $g,h\in\mathcal{N}^+$ such that $fg=1+Bh$.
        \item There exist universal constants  $C_0>0$ and $C >0$ such that the following statement holds. Let $H\in QB_+ (\mathbb{D})$ with $H\geq C_0 H_\Lambda$. Assume that for any $f\in H^\infty$, $||f||_\infty\leq1$ such that
            \begin{equation*}
\label{partb1}                
 |f(\lambda_k)|> e^{-CH(\lambda_k)}, \quad k=1,2, \ldots , 
            \end{equation*}
        there exist $g,h\in\mathcal{N}^+$ such that $fg=1+Bh$. Then, the function $-\log|B|$ has a quasibounded harmonic majorant on the set $\{z\in\mathbb{D}\, : \rho(z,\Lambda)\geq e^{-H(z)}\}$.
    \end{enumerate}
\end{theorem}

Hence, as in the case of the Nevanlinna class, the invertibility problem in $\mathcal{N}^+ / B \mathcal{N}^+$ reduces to study the set of functions $H\in QB_+ (\mathbb{D})$ such that $-\log|B|$ has a quasibounded harmonic  majorant on the set $\{z\in\mathbb{D}\,:\rho(z,\Lambda)\geq e^{-H(z)}\}$. So, given an inner function $I$ with zero set $\Lambda$, it is natural to consider the set $\mathcal{H}_{QB} (I)$ of functions $H\in QB_+ (\mathbb{D})$ such that $-\log|I|$ has a quasibounded harmonic  majorant on the set $\{z\in\mathbb{D}\,:\rho(z,\Lambda)\geq e^{-H(z)}\}$. Our next result says that if $I$ has a non constant singular inner factor, then  $\mathcal{H}_{QB} (I)$ does not contain large functions.

\begin{lemma}
Let $I$ be an inner function with zero set $\Lambda$ and non constant singular inner factor $S$. Then for any $H \in Har_+(\mathbb{D})$ with $H>H_\Lambda$, the function $-\log|I|$ has no quasibounded harmonic majorant on the set $\{z\in\mathbb{D}\,:\rho(z,\Lambda)\geq e^{-H(z)}\}$. 
\end{lemma}
\begin{proof}
 We argue by contradiction. So assume that $H, H_1$ are quasibounded positive harmonic functions such that 
\begin{equation}
\label{final}
- \log |I(z)| \leq H_1(z) \, \text{ if } \,\rho(z,\Lambda)\geq e^{-H(z)} 
\end{equation}
For any $z \in \mathbb{D}$ we apply Lemma 1.1 of \cite{MNT} to get $\tilde z \in \mathbb{D}$ with $\rho (z, \tilde z) < e^{- H(z) / 10}$ and   $\rho ( \tilde z , \Lambda) > e^{- H(z)}$. Hence \eqref{final} gives 
$$
- \log |S(\tilde z)| \leq H_1(\tilde z) .
$$
By Harnack's inequality, there exists an absolute constant $C>0$ such that $-\log |S(z)| \leq C H_1 (z)$. Since $-\log |S|$ is the Poisson integral of a non trivial singular measure on the unit circle, this is a contradiction. 

\end{proof}

If $I$ has a non constant singular inner factor and finitely many or very sparse zeros, the set $\mathcal{H}_{QB} (I)$ is empty. On the other if $I$ satisfies the WEP property then $\mathcal{H}_{QB} (I)$ contains the constants. When $I$ is a Blaschke product, $\mathcal{H}_{QB} (I)$ is the whole cone of positive quasiharmonic functions if and only if the zeros of $I$ are a finite union of interpolating sequences in the Smirnov class. See \cite{MNT}. We have not explored the analogues of our Theorems 2-5 for the class $\mathcal{H}_{QB} (B)$.

 \end{document}